\documentclass[journal]{IEEEtran}
\usepackage{amsmath,amssymb,amsthm} 
\usepackage{bbm}
\usepackage{mathrsfs,bm,enumerate}
\usepackage{graphicx,color,tikz}
\usepackage{caption,subcaption}
\usepackage[T1]{fontenc} 
\usepackage[latin9]{inputenc}
\usepackage{amsthm}
\usepackage{mathtools}
\usepackage{bookmark}

\newcommand{\inR}{\in \mathbb{R}}
\newcommand{\inZ}{\in \mathbb{Z}}
\newcommand{\C}{ \mathbb{C}}
\newcommand{\R}{ \mathbb{R}}

\newcommand{\N}{ \mathbb{N}}

\newcommand{\eqdef}{\stackrel{\vartriangle}{=}}

\newcommand{\Dop}{{\rm D}}

\newcommand{\dint}{{\rm d}}

\newcommand{\Fourier}{ \mathcal{F}}

\DeclareMathOperator*{\esssup}{ess\,sup}

\providecommand{\revise}[1]{{#1}}
\providecommand{\Revise}[1]{{#1}}

\def\Spc#1{{\mathcal{#1}}}  
\def\Op#1{{\mathrm{#1}}}  
\def\ee{\mathrm{e}} 
\def\jj{\mathrm{j}} 
\def\Indic{\mathbbm{1}} 
\def\One{\mathbbm{1}}

\newcommand{\embedC}{\xhookrightarrow{}}

\newcommand{\embedIso}{\xhookrightarrow{\mbox{\tiny \rm iso.}}}
\newcommand{\toC}{\xrightarrow{\mbox{\tiny \ \rm c.  }}}

\renewcommand{\[}{\begin{equation}}
\renewcommand{\]}[1]{\label{eq:#1}\end{equation}}

\newtheorem{definition}{Definition}
\newtheorem{proposition}{Proposition}

\newtheorem{theorem}{Theorem}

\title{A Note on BIBO Stability\thanks{The research leading to these results has received funding from the Swiss National Science
Foundation under Grant 200020-162343/1.
}}

\author{
Michael Unser\thanks{Biomedical Imaging Group, \'Ecole polytechnique f\'ed\'erale de Lausanne (EPFL),
Station 17, CH-1015 Lausanne, Switzerland ({\tt michael.unser@epfl.ch}). }
 }

\begin{document}

\maketitle

\begin{abstract} 
The statements on the BIBO stability of continuous-time convolution systems found in engineering textbooks are often either too vague (because of lack of hypotheses) or mathematically incorrect. What is more troubling is that they usually exclude the identity operator. The purpose of this note is to clarify the issue while presenting some fixes. In particular, we show that a linear shift-invariant system
is BIBO-stable in the $L_\infty$-sense if and only if its impulse response is included in the space of bounded Radon measures, which is a superset of $L_1(\R)$ (Lebesgue's space of absolutely integrable functions). As we restrict the scope of this characterization to the convolution operators whose impulse response is a measurable function, we recover the classical statement.
\end{abstract}

\section{Introduction}
A statement that is made in most courses on the theory of linear systems as well as in the English version of Wikipedia\footnote{ \url{https://en.wikipedia.org/wiki/BIBO_stability}. {\em Accessed November 2019}.} is that a convolution operator is stable in the BIBO sense (bounded input and bounded output) if and only if its impulse response is absolutely summable/integrable.
While the proof of this equivalence is fairly straightforward for discrete-time systems, there seems to be some confusion in the continuous domain (see Appendix B for specific references), especially since the above statement excludes the identity operator, whose impulse response is the Dirac distribution $\delta$. Since $\delta$ is not a measurable function in the sense of Lebesgue (see explanations in Appendix A) and hence not a member of $L_1(\R)$, does this mean that the identity operator is not BIBO-stable? Obviously not; this is what we want to clarify here. The argument, which is somewhat technical, rests on the shoulders of two giants: Laurent Schwartz and Lars H\"ormander, who were awarded the Fields medal in 1950 and 1962, respectively, for their fundamental contributions to the theory of distributions and partial differential equations. 


\revise{In the sequel, we shall revisit the topic of BIBO stability with the help of appropriate mathematical tools. In Section \ref{Sec:Classical}, we recall the classical integral definition of a convolution operator. 
We then present a correction to the standard characterization of BIBO-stable filters (Proposition \ref{Theo:BIBO}) together with a new upgraded proof. Since the underlying assumption that the impulse response should be a measurable function excludes the identity operator, we first explain in Section III the extended (distributional) form of convolution supported by Schwartz' kernel theorem (Theorem \ref{Theo:SchwartzLSI}). Based on this formalism, we present two Banach-space extensions
of the classical result that should settle the issue: a first one (Theorem \ref{Theo:BIBO2}) that imposes that the result of the convolution be continuous, and a second (Theorem \ref{Theo:BIBO3}) that characterizes the BIBO-stable filters in full generality. 
The mathematical derivations are presented in Section \ref{Sec:Distribution}, where we also make the connection with known results in harmonic analysis.
}

We like to mention a similar clarification effort by Hans Feichtinger, who proposes to limit the framework to convolution operators that are operating on $C_0(\R)$ (a well-behaved subclass of bounded functions) in order to avoid pathologies \cite{Feichtinger2017b}.
This is another interesting point of view that is complementary to ours, as discussed in Section \ref{Sec:Distribution}.
\section{BIBO Stability: The Classical Formulation}
\label{Sec:Classical}
The convolution 
of two functions $h, f: \R \to \R$ is the function \revise{usually} specified by
\begin{align}
\label{Eq:Conv}
t \mapsto (h \ast f)(t)\eqdef \int_\R h(\tau) f(t-\tau) \dint \tau\end{align}
under the implicit assumption that the integral in \eqref{Eq:Conv} is well defined for any $t\inR$. \revise{This latter point will be clarified as we develop the mathematics.}
In particular, this requires
that the functions $f$ and $h$ both be measurable\footnote{ A function $f: \R \to \R$ is said to be Lebesgue-measurable if the preimage $f^{-1}(E)$ of any Borel set $E$ in $\R$ is a Borel set \cite{Rudin1987}. 
 The property is preserved through pointwise multiplication and translation.} in the sense of Lebesgue. 
Here, instead of designating the continuous-time signal by $f(t)$ and its convolved (or filtered) version
 by $h(t) \ast f(t)$, as engineers usually do, we are using the less ambiguous mathematical notations
 $t \mapsto f(t)$ or $f \in L_p(\R)$ and $t \mapsto (h \ast f)(t)$ or $h \ast f \in L_\infty(\R)$.

If we fix $h$ and consider $f$ as the input signal, then \eqref{Eq:Conv}
defines a linear shift-invariant (LSI) operator (or system) denoted by
$\Op T_h: f \mapsto h\ast f$. Its impulse response $h$ is then \revise{formally} described as
$h=\Op T_h\{\delta\}$, where $\delta \in \Spc D'(\R)$ is the Dirac distribution and $\Spc D'(\R)$  Schwartz' space of distributions \cite{Schwartz:1966}. \revise{ This interpretation is backed by Schwartz' kernel theorem, as explained in Section \ref{Sec:ExtendedConvolution}.}

An important practical requirement for an
LSI system is that its response to any bounded input remains bounded.
 \Revise{There is one mathematical aspect, however, that makes the formulation of BIBO stability nontrivial in the continuous domain: Depending on the context,} the input and output boundedness requirements 
can be strict, with $\revise{\|f\|_{L_\infty}}=\|f\|_{\sup}\eqdef \sup_{t\in \R} |f(t)|<\revise{\infty}$, which arises when the function $f$ is continuous \revise{(i.e.,  $f \in C(\R)$)}, or in the looser sense of Lebesgue: $|f(t)|\le \|f\|_{L_\infty}<\infty$ for almost any $t\inR$ (see \revise{Section \ref{Sec:BanachBounded} for additional explanations}).
\Revise{
This latter condition is often expressed as $f \in L_\infty(\R)$ where
$L_\infty(\R)
=\{f: 
\R \to \R \ \ \mbox{s.t.} \ \ 
f \mbox{ is measurable and } \|f\|_{L_\infty}<\infty
\}
$ is Lebesgue's space of bounded functions.
} 
\begin{definition} 
\label{Def:BIBO} 
The linear operator $\Op T: f \mapsto \Op T\{f\}$ is said to be BIBO-stable if
\begin{enumerate}
\item \Revise{$ \Op T\{f\}$ is well-defined for any $f\in L_\infty(\R)$, and,} 
\item there exists a constant $C>0$ independent of $f$ such that
$$
\|\Op T\{f\}\|_{L_\infty} \le C \|f\|_{L_\infty}
$$
for all $f\in L_\infty(\R)$.
\end{enumerate}
\end{definition}

%
%
%
%

The standard condition for the BIBO stability of the continuous-time  convolution operator $\Op T_h: f \mapsto h\ast f$ 
that is found in engineering textbooks is $\|h\|_{L_1}<\infty$,
where the $L_1$-norm is defined by
\begin{align}
\label{Eq:L1}
\|h\|_{L_1}\eqdef \int_\R |h(t)| \dint t.
\end{align}
A slightly more precise 
 statement is $h\in L_1(\R)$,
where
$$
L_1(\R)=\{f: 
\R \to \R \ \ \mbox{s.t.} \ \ 
f \mbox{ is measurable and } \|f\|_{L_1}<\infty
\}
$$
is Lebesgue's space of absolutely integrable functions. 

The sufficiency of the condition $h\in L_1(\R)$ is deduced from the standard estimate
 \begin{align*}
\left|\int_{\R}h(\tau) f(t-\tau) \, \dint \tau \right| & \leqslant \int_{\R}  |h(\tau)|\cdot|f(t-\tau)|\, \dint\tau \\
&
\leqslant \left(\int_{\R}\left| h(\tau)\right|   \dint \tau\right)\, \|f\|_{L_\infty},
\end{align*}
which is valid for any $t\inR$. The convolution integral 
\eqref{Eq:Conv} is therefore well defined if $f\in L_\infty(\R)$, which then also yields the classical bound on BIBO stability\begin{align}
\label{Eq:supinequal}
\revise{\|h \ast f \|_{L_\infty} }\le \|h\|_{L_1} \, \|f\|_{L_\infty}< \infty.
\end{align}
By adapting the argument that is used in the discrete-time formulation of BIBO stability, many authors (see Appendix B) claim that the condition $h \in L_1(\R)$ is also necessary. To that end,  they apply the convolution system to a ``worst-case'' signal 
\begin{align}
\label{Eq:f0}
f_0(t)={\rm sign}\big(h(-t)\big)
\end{align}
in order to produce the strongest response at $t=0$,
\begin{align*}
(h \ast f_0)(0) =   \int_{-\infty}^{+\infty}  h(\tau){\rm sign}\big(h(\tau)\big)\,  \dint \tau= \int_{-\infty}^{+\infty} \left| h(\tau)\right| \dint \tau,
\end{align*}
 \Revise{which is then claimed to saturate the stability bound \eqref{Eq:supinequal}
with $
\|h \ast f_0\|_{L_\infty}=\|h\|_{L_1} \, \|f_0\|_{L_\infty}$. 
Unfortunately, this simple reasoning has two shortcomings. 
First, unlike in the discrete setting, the characterization of what happens at $t=0$, which is a set of measure zero, is not sufficient to deduce  that
$\|h \ast f_0\|_{L_\infty}\ge (h\ast f_0)(0)$, unless one invokes the continuity of 
$t \mapsto (h \ast f_0)(t)$, which is not yet known at this stage (see Theorem \ref{Theo:BIBO2}). Second, one cannot ensure that
the Lebesgue convolution integral \eqref{Eq:Conv} is well defined for $f_0\in L_\infty(\R)$, unless $h$ is Lebesgue-integrable\footnote{Any measurable function $h: \R \to \R$ admits a unique decomposition as $h=h^+ - h^-$ with $h^+, h^{-}: \R \to \R_{\ge0}$. It is 
Lebesgue integrable if $\min(\|h^+\|_{L_1},\|h^-\|_{L_1})<\infty$ \cite{Folland2013}.}, which then considerably limits the scope of the claim about necessity. 
}


Our first practical fix 
is an extension of the argumentation to
the larger space $L_{1,\rm loc}(\R)$ of measurable functions that are {\em locally} 
integrable, meaning that $\int_\mathbb{K} |h(t)| \dint t< \infty$ over any compact domain $\mathbb{K} \subset \R$.
 The reassuring outcome, which conforms with the practice in the field, is that one can determine the stability of an LSI system by integrating the absolute value of its impulse response---even if $h$ is not {\em globally} Lebesgue integrable, as in the case of an increasing and possibly oscillating exponential.

\begin{proposition}
\label{Theo:BIBO}
If $h \in L_1(\R)$, then the convolution operator $f \mapsto h \ast f$ defined by \eqref{Eq:Conv} is BIBO-stable with $\|h \ast f\|_{L_\infty}\le \|h\|_{L_1}\|f\|_{L_\infty}$.
Conversely, if the impulse response $h$ is measurable and locally integrable with $\int_\R |h(t)|\dint t 
=\revise{\infty}$, then the system is not BIBO-stable, in which case it is said to be unstable.
\end{proposition}
\begin{proof} The first statement is a paraphrasing of \eqref{Eq:supinequal}. For the converse part, we 
assume that  $h \in L_{1,\rm loc}(\R)$ with $\int_\R |h(t)|\dint t=\revise{\infty}$.
Because of the local integrability of $h$, one can then still rely on the definition of the convolution given by \eqref{Eq:Conv}, but only if the input function $f$ is bounded and compactly supported.
By considering the truncated versions $f_{0,T}=f_0 \cdot \Indic_{[-T,T]}$ of the worst-case signal  \eqref{Eq:f0}, we can therefore determine the maximal value of the output signal as
$$
(h \ast f_{0,T})(0)= \int_{-T}^{+T} h(\tau) {\rm sign}\big(h(\tau)\big) \dint \tau =  \int_{-T}^{+T} |h(\tau)| \dint \tau.
$$
\Revise{The additional ingredient is the continuity of $t\mapsto (h \ast f_{0,T})(t)$ in the neighborhood of $t=0$ (see Proposition \ref{Prop:Loc} in Appendix D), 
 which allows us to conclude that $(h \ast f_{0,T})(0)\le \sup_{t \in \R}|h \ast f_{0,T}(t)|=\|h \ast f_{0,T}\|_{L_\infty}$.}
While the latter quantity is finite for any fixed value of $T$, 
we have that $\lim_{T \to \infty} (h \ast f_{0,T})(0) =\int_\R |h(t)| \dint t=\revise{\infty}$, which indicates that the output signal becomes unbounded in the limit. This shows that the underlying system is unstable. 

\end{proof}
Another way of obtaining Proposition \ref{Theo:BIBO} is as a corollary of Theorem \ref{Theo:BIBO3} (the complete characterization of BIBO-stable systems) and Proposition \ref{Theo:BIBOcriterion} in Section IV.
The important examples of unstable filters that fall within the scope of Proposition \ref{Theo:BIBO} are the systems ruled by differential equations with at least one pole in the right-half complex plane; for instance, 
$h(t)=\Indic_{+}(t) \ee ^{\alpha t}$ with ${\rm Re}(\alpha)\ge 0$ \cite{Oppenheim:1996}. The derivative operator with $h=\delta'$ and the Hilbert transform with $h(t)=1/(\pi t)$ are unstable as well (as asserted by Theorem \ref{Theo:BIBO3}), but they fall outside the scope of Proposition \ref{Theo:BIBO}:
the first because $\delta'$ is not a function (but a distribution), and the second because the function $1/t$ is not locally integrable---in fact, the impulse response of the Hilbert transform is the distribution ``$1/(\pi t)$'' that requires the use of a ``principal value'' for the proper definition of the convolution integral \cite{Stein1971}.

In the stable scenario, where $h \in L_1(\R)$, we are able to characterize the underlying filter by its frequency response 
\begin{align}
\widehat{h}(\omega)\eqdef \Fourier\{h\}(\omega)=\int_\R h(t) \ee^{-\jj \omega t}\dint t,
\end{align}
which is the ``classical'' Fourier transform of $h$.
Moreover, the Riemann-Lebesgue lemma ensures that
$\widehat{h} \in C_0(\R)$ with $\|\widehat{h}\|_{\sup} \le \|h\|_{L_1}$.
We recall that $C_0(\R)$ is the Banach space of continuous and bounded functions that decay at infinity, equipped with the $\sup$-norm.

\section{Banach Formulations of BIBO Stability}
\label{Sec:Banach}
The classical textbook statements on continuous-time BIBO stability, including our reformulation in Proposition \ref{Theo:BIBO}, have two limitations.
First, they exclude the identity operator with $h=\delta$, as explained in Appendix A.
Second, they are 
often evasive 
concerning the 
 hypotheses under which the condition $h \in L_1(\R)$ is necessary (see Appendix B). In this section, we show how this can be corrected by considering appropriate Banach spaces.

\subsection{Extension of the Notion of Convolution}
\label{Sec:ExtendedConvolution}
The scope of our mathematical statements relies on Schwartz' famous kernel theorem \revise{\cite{Gelfand-Villenkin1964, Schwartz1950} which delineates the complete 
 class of linear operators} that continuously map $\Spc D(\R) \to \Spc D'(\R)$.
We recall that $\Spc D(\R)=\Spc C_{\rm c}^\infty(\R)$ is the space of smooth and compactly supported test functions equipped with the usual Schwartz topology\footnote{A sequence of functions $\varphi_k \in \Spc D(\R)$ is said to converge to $\varphi$ in $\Spc D(\R)$ if (i) there exists a compact domain $\mathbb{F}$ that includes the support of $\varphi$ and of all $\varphi_k$, and (ii) 
$\|\varphi_k-\varphi\|_{n}\to 0$ for all $n\in \N$,
where $\|\varphi\|_n\eqdef \|\Dop^n\varphi\|_{L_\infty}$ with
$\Dop^n: \Spc D(\R) \to \Spc D(\R)$ the $n$th derivative operator.}. Its topological dual
$\Spc D'(\R)$ is the space of generalized functions also known as {distributions}. 
In essence, a distribution $f \in \Spc D'(\R)$ is a linear map---more precisely, a continuous linear functional---that assigns a real number to each test function $\varphi \in \Spc D(\R)$; this is denoted by $f: \varphi \mapsto \langle f,\varphi \rangle$. For instance, the definition of Dirac's impulse as a distribution is $\delta: \varphi \mapsto \langle \delta,\varphi \rangle\eqdef \varphi(0)$.

\revise{Beside linearity, the property that defines an LSI operator is 
$\Op T_{\rm LSI}\{\varphi(\cdot-t_0)\}(t)=\Op T_{\rm LSI}\{\varphi\}(t-t_0)$ for any $t_0\inR$. Schwartz' theorem then tells us that there is a one-to-one correspondence between
continuous LSI operators $\Spc D(\R) \to C(\R)$ and distributions, with the 
defining distribution $h \in \Spc D'(\R)$ being the impulse response of the operator. 
The relevant space of continuous functions  here is $C(\R)$ with the topology of uniform convergence over compact sets, which involves the system of seminorms $\|f\|_N=\sup_{|t|\le N} |f(t)|, N\in \N$. The latter is an extended functional setup that tolerates arbitrary growth at infinity.

}
%
%
%
%
\revise{
\begin{theorem}[Schwartz' kernel theorem for LSI operators]
\label{Theo:SchwartzLSI}
For any given $h\in \Spc D'(\R)$, the operator $\Op T_h: \varphi \mapsto h \ast \varphi$
with
\begin{align}
\label{Eq:ConvD}
t \mapsto (h \ast \varphi)(t)\eqdef \langle h, \varphi(t-\cdot)\rangle
\end{align}
is LSI and continuously maps $\Spc D(\R) \toC C(\R)$. 
Conversely, for every LSI operator $\Op T_{\rm LSI}: \Spc D(\R) \toC C(\R)$, there is a unique $h\in \Spc D'(\R)$ such that
$\Op T_{\rm LSI}=\Op T_{h}: \varphi \mapsto h \ast \varphi$ where the convolution is specified by
\eqref{Eq:ConvD}.
\end{theorem}
}


Then, depending on the decay properties of $h$, it is generally possible to extend the domain of the convolution operator $\Op T_h$ to some appropriate Banach space according to the procedure described in 
Section \ref{Sec:Distribution}.
For instance, if $h \in L_1(\R)$, then $\Op T_h$ has a continuous extension $L_\infty(\R) \to C_{\rm b}(\R) \subset  C(\R)$ that coincides with the classical definition given by \eqref{Eq:Conv}.

\revise{Finally, we note that, for the cases where the Dirac impulse $\delta$ is in the domain of the extended operator (for instance, when $h \in C(\R)$),  the distributional definition of the convolution given by \eqref{Eq:ConvD} yields $\Op T_h\{\delta\}=h\ast \delta = h$, which explains the term ``impulse response.''
}

\subsection{Banach Spaces of Bounded Functions}
\label{Sec:BanachBounded}
In order to investigate the issue of BIBO stability, it is
helpful to describe the boundedness and continuity properties of functions via their inclusion in appropriate Banach subspaces of $\Spc D'(\R)$. The three relevant function spaces are
$$
C_0(\R) \subset C_{\rm b}(\R)  \subset L_\infty(\R).
$$
The central space consists of the subset of bounded functions that are continuous:
\begin{align*}
C_{\rm b}(\R)=\left\{f: 
\R \to \R \  \mbox{ s.t. } 
f \mbox{ is continuous and } \|f\|_{\sup}<\infty
\right\}.
\end{align*}
It is a classical example of Banach space---a complete normed vector space \cite{Megginson1998}.
The smaller space $C_0(\R)$, which is also equipped with the ${\sup}$-norm, imposes the additional constraint that $f(t)$ should vanish at $t=\pm\infty$.
It is best described as the completion of $\Spc D(\R)$ equipped with the $\sup$-norm, which will have its importance in the sequel. This property is indicated by $C_0(\R)=\overline{(\Spc D(\R), \|\cdot\|_{\sup})}$. The concept is also valid for
$L_1(\R)$, which can be described as $L_1(\R)=\overline{(\Spc D(\R), \|\cdot\|_{L_1})}$, where 
the $L_1$-norm is defined by \eqref{Eq:L1} with $f \in \Spc D(\R)$ and the integral being classical---in the sense of Riemann.
This completion property applies to $L_p(\R)=\overline{(\Spc D(\R), \|\cdot\|_{L_p})}$
with $p\in[1,\infty)$ as well \cite[Proposition 8.17, p. 254]{Folland2013}, but not for $p=\infty$, which explains the importance of the space $C_0(\R)$, which is distinct from $L_\infty(\R)$.

%
%
%

In order to properly identify $L_\infty(\R)$ as a subspace of $\Spc D'(\R)$, we shall exploit the property that the $L_\infty$-norm is the dual of the $L_1$-norm.
We therefore choose to define the $L_\infty$-norm as
\begin{align}
\label{Eq:Linftydual}
\|f\|_{L_\infty} \eqdef \sup_{ \varphi \in \Spc D(\R):\,  \|\varphi\|_{L_1}\le 1}\langle f, \varphi\rangle=\sup_{  \varphi \in L_1(\R):\,  \|\varphi\|_{L_1}\le 1}
\langle f, \varphi\rangle,
\end{align}
where the central part of \eqref{Eq:Linftydual} takes advantage of the denseness\footnote{This means that, for any $f \in L_1(\R)$ and any $\epsilon>0$, there exists a function $\varphi_\epsilon \in \Spc D(\R)$ such that $\|f -\varphi_\epsilon\|_{L_1}<\epsilon$. It is a direct consequence of $L_1(\R)=\overline{(\Spc D(\R), \|\cdot\|_{L_1})}$.}
 of $\Spc D(\R)$ in $L_1(\R)$.
This yields a definition that is valid not only for (measurable) functions, but also for all
$f \in \Spc D'(\R)$. 
Consequently, we can redefine our target space as
\begin{align}
\label{Eq:Linftyl}L_\infty(\R)=\{f\in \Spc D'(\R): \|f\|_{L_\infty}<\infty\},
\end{align}
which is readily identified as the topological dual of $L_1(\R)$; that is, $L_\infty(\R)=\big(L_1(\R)\big)'$ due to the dual specification of the $L_\infty$-norm given by the right-hand side of \eqref{Eq:Linftydual}.

While \eqref{Eq:Linftyl} defines $L_\infty(\R)$ as a subspace of $\Spc D'(\R)$, we can also
identify its elements as (bounded) measurable functions $f: \R \to \R$ via the classical association
\begin{align}
\label{Eq:dualityprod}
\varphi \mapsto \langle f, \varphi  \rangle\eqdef \int_\R \varphi(t)f(t) \dint t,
\end{align}
where the right-hand side of \eqref{Eq:dualityprod} is a standard Lebesgue integral.
\revise{Now, the main difference between the $\sup$-norm and the $L_\infty$-norm
is that, for $f \in L_\infty(\R)$ (identified as a function), the inequality $|f(t)| \le \|f\|_{L_\infty}$ holds for almost every $t\inR$. This means that it holds over the whole real line except, possibly, on a set of measure zero. This is often indicated as $\|f\|_{\infty}=\esssup_{t \inR} |f(t)|$, using the notion of essential supremum. In other words, the $L_\infty$-norm is more permissive than the $\sup$-norm with $\|f\|_{\infty}\le \|f\|_{\sup}$. However, the two norms are equal whenever the function $f$ is continuous, which translates into the isometric inclusion $C_{\rm 0}(\R)\embedIso C_{\rm b}(\R)\embedIso L_\infty(\R)$.}

\subsection{Extended Results on BIBO Stability}
Remarkably, the combination of the two latter function spaces enables us to formulate a first Banach extension of the classical statement on BIBO stability.
\revise{To that end, we shall restrict the distributional 
framework covered by Theorem \ref{Theo:SchwartzLSI} to the case where
the impulse response $h$ is identifiable as a measurable function \big(i.e., $h \in L_{1,{\rm loc}}(\R)\subset \Spc D'(\R)$\big). The linear functional on the right-hand side of \eqref{Eq:ConvD} then has an explicit integral description given by \eqref{Eq:Conv}
with $f=\varphi \in \Spc D(\R)$.
Within this class of convolution operators, we now identify the ones whose domain can be extended to $L_\infty(\R)$.}

\begin{theorem}
\label{Theo:BIBO2} 
The convolution operator $\Op T_h: f \mapsto h \ast f$ with $h \in L_{1,\rm loc}(\R)$ 
\revise{has a continuous extension} $L_\infty(\R) \toC C_{\rm b}(\R)$ if and only if $h \in L_1(\R)$.
Moreover,
$$
\|h \ast f\|_{\sup} \le \|h \|_{L_1} \|f\|_{L_\infty}
$$
with the bound being sharp in the sense that it also yields the norm of the underlying operator: $\|\Op T_h\|_{L_\infty\to C_{\rm b}}=\|h\|_{L_1}$ (see Definition \ref{Def:OpNorm}).
\end{theorem}


The proof of this result is deferred to Section \ref{Sec:Distribution} \big(see Item 2) and the final statement of Theorem \ref{Theo:Convolution}.

\revise{It is of interest to compare Proposition \ref{Theo:BIBO} and  Theorem \ref{Theo:BIBO2} because they address the problem of stability from 
different but complementary perspectives.
Proposition \ref{Theo:BIBO} is focused primarily on the well-posedness of the convolution integral \eqref{Eq:Conv} for $f \in L_\infty(\R)$. It can be paraphrased as: ``Let $h$ be a measurable (and locally integrable) function. Then, the Lebesgue integral \eqref{Eq:Conv} defines a convolution operator that is BIBO-stable if and only if $h\in L_1(\R)$.''
By contrast,
Theorem \ref{Theo:BIBO2} considers the complete family of ``classical'' convolution operators 
$\Op T_h: \Spc D(\R) \to C(\R)$ with $h \in L_{1,\rm loc}(\R)$ and precisely identifies the subset of operators that have a continuous extension from $L_\infty(\R) \to C_{\rm b}(\R)$.
Since  $C_{\rm b}(\R)$ is isometrically embedded in $L_\infty(\R)$, this is more informative than Proposition \ref{Theo:BIBO}
because it also tells us
that 
$(h\ast f)(t)$ is a continuous function of $t\inR$. In that respect,
we note that the requirement that the convolution of any bounded function $f$ be continuous excludes the use of the identity operator with $h=\delta$ at the onset, even if we extend the framework to $h\in \Spc D'(\R)$.
}
%
%
%

To obtain a more permissive characterization of BIBO stability, we need 
to extend the range of the operator from $C_{\rm b}(\R)$ to $L_\infty(\R)$, which should then also translate into a corresponding enlargement of the class of admissible impulse responses. We shall delineate the latter in a way that parallels our definition of $L_\infty(\R)$, with the roles of the $L_1$- and $\sup$- (or $L_\infty$-) norms being interchanged. 
To that end, we first define the $\Spc M$-norm as
\begin{align}
\label{Eq:Mdual}
\|f\|_{\Spc M} \eqdef \sup_{ \varphi \in \Spc D(\R):\, \|\varphi\|_{\sup}\le 1} 
\langle f, \varphi\rangle.
\end{align}
This then yields the Banach space
\begin{align}
\Spc M(\R)=\{f\in \revise{\Spc D'(\R)}: \|f\|_{\Spc M}<\infty\},
\end{align}
which also happens to be the space of bounded Radon measures\footnote{We adhere with Bourbaki's nomenclature to distinguish the two complementary interpretations of a measure:
either as a continuous linear functional on $\Spc D(\R)$ (Radon measure), or as a set-theoretic additive rule that associates a real number to any Borel set of $\R$ (signed Borel measure) \cite{Bourbaki2004, Bony2001}. 
} on $C_0(\R)$.
In other words, $\Spc M(\R)$ is the topological dual of $C_0(\R)$. Moreover, we can invoke the Riesz-Markov theorem to identify $\Spc M(\R)=\big(C_0(\R)\big)'$ with the space of bounded signed Borel measures on $\R$ \cite{Rudin1987}. Concretely, this means that any $h\in \Spc M(\R)$ is associated with a unique Borel measure $\mu_h$, which then gives a concrete definition of the linear functional
\begin{align}
\label{Eq:RadontoBorel}
f \mapsto \langle h,f \rangle \eqdef \int_\R f(\tau) \dint \mu_h(\tau)
\end{align}
for any measurable function $f$, while the total-variation norm of the measure $\mu_h$ is given by
$\|\mu_h\|_{\rm TV}\eqdef\int_\R \dint|\mu_h|=\|h\|_{\Spc M}$ (see Section \ref{Sec:Radon}).
The main point for us is that $\Spc M(\R)$ is a superset of $L_1(\R)$, with $\|f\|_{\Spc M}=\|f\|_{L_1}$ for all 
$f \in L_1(\R)$. Moreover, we have that
$\delta(\cdot-t_0) \in \Spc M(\R)$ for any $t_0\inR$ with $\|\delta(\cdot-t_0)\|_{\Spc M}=1$, as can be readily inferred from \eqref{Eq:Mdual} by considering a \Revise{non-negative} test function that achieves its maximum $\varphi(t_0)=1$ at $t=t_0$.

For the cases where the impulse response $h\in \Spc M(\R)$ is not an $L_1$ function, we  extend our definition of the original (Lebesgue) convolution
integral as 
\begin{align}
\label{Eq:ConvMeasure}
t \mapsto (h \ast f) (t)= \langle h, f(t-\cdot)\rangle\eqdef \int_\R f(t-\tau) \dint \mu_h(\tau),
\end{align}
which is the same as \eqref{Eq:Conv} when we can write $\dint\mu_h(\tau)=h(\tau)\dint \tau$, which happens when the corresponding measure $\mu_h$ is absolutely continuous\footnote{Another way to put it is that $h$ is the Radon-Nikodym derivative of $\mu_h$.} with respect to the Lebesgue measure.
A standard manipulation then yields that
\begin{align}
\label{Eq:Minequal}
\left|(h \ast f) (t)\right|& \le \int_\R |f(t-\tau)| \, \dint |\mu_h|(\tau) \nonumber\\
&\le \|f\|_{L_\infty} \int_\R \dint |\mu_h| =\|f\|_{L_\infty} \|h\|_{\Spc M},
\end{align}
which is the basis for the direct (easy) part of Theorem \ref{Theo:BIBO3}, where the complete class of BIBO-stable systems is identified, including the identity operator. 

\begin{theorem} 
\label{Theo:BIBO3}
The convolution operator $\Op T_h: f \mapsto h \ast f$ with $h \in \Spc D'(\R)$ 
\Revise{has a continuous extension}
$L_\infty(\R) \toC L_\infty(\R)$ 
if and only if $h \in \Spc M(\R)$.
Moreover,
$$
\|h \ast f\|_{\infty} \le \|h \|_{\Spc M} \; \|f\|_{L_\infty}
$$
with the bound being sharp in the sense that $\|\Op T_h\|_{L_\infty\to L_\infty}=\|h\|_{\Spc M}$.
\end{theorem}
%

This result, which is also valid in dimensions higher than $1$, is known in harmonic analysis \cite[p. 140 Corollary 2.5.9]{Grafakos2004},
\cite{Stein1971}
but much less so in engineering circles. It can be traced back to an early paper by H\"ormander that provides a comprehensive treatment of convolution operators on $L_p$ spaces \cite{Hormander1960}. 
The reminder of the paper is devoted to the proof of the two theorems on BIBO stability and 
of some
interesting variants (see Theorem \ref{Theo:Convolution}). To that end, we shall rely on Schwartz' powerful distributional formalism which, as we shall see, allows for a rather soft derivation, once the prerequisites have been laid out.

\section{Mathematical Derivations}
\label{Sec:Distribution}
 
\subsection{Extension of Convolution Operators}
The most general form of a convolution operator 
backed by Schwartz' kernel theorem \revise{(see Theorem \ref{Theo:SchwartzLSI})}
is $\Op T_h: \Spc D(\R) \to C(\R)\embedC \Spc D'(\R)$ with $h \in \Spc D'(\R)$, where 
$\Op T_h\{\varphi\}$ is defined by \eqref{Eq:ConvD} for any $\varphi \in \Spc D(\R)$.
The two complementary ingredients at play there
 are: (i) the restriction of the domain
to $\Spc D(\R)$---the ``nicest'' class of functions in terms of smoothness and decay---and (ii) the extension of the range to 
$\Spc D'(\R)$, which can accommodate an arbitrary degree of growth (polynomial, or even exponential) at infinity. In other words, the theoretical framework is such that it can deal with the very worst
scenarios, including unstable differential systems whose impulse response is exponentially increasing.
 
Then, depending on the smoothness and decay properties of $h$, it is usually possible
to extend the domain of $\Op T_h$ to some Banach space $\Spc X\supseteq \Spc D(\R)$ that is continuously embedded in $\Spc D'(\R)$, which is denoted by
$\Spc X \embedC \Spc D'(\R)$. For this to be feasible, we require that
$\|\cdot\|_\Spc X$ be a valid norm on $\Spc D(\R)$ and that $\Spc D(\R)$ be dense in $\Spc X$, which is equivalent to
$\Spc X=\overline{(\Spc D(\R),\|\cdot\|_{\Spc X})}$. In other words, $\Spc X$ is the completion of
$\Spc D(\R)$ equipped with the $\|\cdot\|_{\Spc X}$-norm.

We start by recalling the definition of the norm of a bounded operator.
\begin{definition} 
\label{Def:OpNorm}
Let $(\Spc X,\|\cdot\|_{\Spc X})$ and $(\Spc Y,\|\cdot\|_{\Spc Y})$ be two Banach spaces and $\Op T$ a linear operator $\Spc X \to \Spc Y$. Then, the operator is said to be bounded if
$$
\|\Op T\|_{\Spc X \to \Spc Y}\eqdef \sup_{f \in \Spc X\backslash\{0\}} \frac{\|\Op T\{f\}\|_{\Spc Y} }{\|f\|_{\Spc X}}<\infty.
$$
\end{definition}

A direct consequence of Definition \ref{Def:OpNorm} is that a bounded operator $\Op T: \Spc X \to \Spc Y$ continuously maps $\Spc X$ into $\Spc Y$, as indicated by $\Op T: \Spc X \toC \Spc Y$.

Theorem \ref{Theo:BLTExtension} then describes a functional mechanism that allows us to extend an operator initially defined on $\Spc D(\R)$. It is a particularization of a fundamental extension theorem in the theory of Banach spaces
\cite[Theorem I.7, p. 9]{Reed1980}.

\begin{proposition} [Extension of a linear operator]
\label{Theo:BLTExtension}
\Revise{Let $\Spc X$ and $\Spc Y$ be two Banach subspaces of $\Spc D'(\R)$ with the additional property that $\Spc D(\R)$ is dense in $\Spc X$.
Then, the linear operator $\Op T: \Spc D(\R) \toC \Spc D'(\R)$ }has a unique continuous extension $\Spc X=\overline{(\Spc D(\R),\|\cdot\|_\Spc X)} \toC \Spc Y$ with $\|\Op T \|_{\Spc X \to \Spc Y}\le C$ if and only if 
\begin{align}
(i) \ \ & \Op T\{\varphi\}\in \Spc Y,\quad \mbox{and}\\
(ii) \ \ &\|\Op T\{\varphi\}\|_{\Spc Y} \le C \|\varphi\|_{\Spc X}
\end{align}
for all $\varphi \in \Spc D(\R)$ and some constant $C>0$.
\end{proposition}

Since a convolution operator $\Op T_h: \Spc D(\R) \toC \Spc D'(\R)$ 
is uniquely characterized by its impulse response $h\in\Spc D'(\R)$, the same holds true for its extension $\Op T_h: \Spc X \toC \Spc Y$, which justifies the use of the same symbol. 
Rather than defining $\Op T_h\{f\}=h \ast f$ through a Lebesgue integral as in \eqref{Eq:Conv} or \eqref{Eq:ConvMeasure}, we can therefore rely on \eqref{Eq:ConvD} and
define our extended convolution operator $\Op T_h: \Spc X \to \Spc Y$ through a limit process. Specifically,
we pick a Cauchy sequence $(\varphi_n)$ in $(\Spc D(\R),\|\cdot\|_{\Spc X})$ such that 
$\lim_{n\to\infty} \varphi_n=f\in \Spc X$. Then, the sequence of functions $(g_n=h \ast \varphi_n)$ with
\begin{align}
\label{Eq:ConvS}
t \mapsto (h \ast \varphi_n)(t)=\langle h, \varphi_n(t-\cdot)\rangle
\end{align}
is Cauchy in $\Spc Y$ and converges to a limit $g=\lim_{n\to \infty}(h \ast \varphi_n) \in \Spc Y$,
independently of the choice of the $\varphi_n$ since the space $\Spc Y$ is complete.
We now recapitulate this process in the form of a definition.
\begin{definition}[Banach extension of a distributional convolution operator]
\label{Def:Convolution}
Let $\Spc X$ and $\Spc Y$ be two Banach subspaces of $\Spc D'(\R)$ with the additional property that $\Spc D(\R)$ is dense in $\Spc X$. 
When the two conditions in Theorem \ref{Theo:BLTExtension} hold, the unique continuous extension $\Op T_h: \Spc X \toC \Spc Y$ of the convolution operator specified by \eqref{Eq:ConvS} with $h \in \Spc D'(\R)$ is defined by
\begin{align}
\label{Eq:ConvS2}
\Op T_h: f \mapsto h \ast f \eqdef \lim_{n\to \infty}(h \ast \varphi_n) \in \Spc Y,
\end{align}
where $(\varphi_n)$ is any sequence in $\Spc D(\R)$ such that 
$\lim_{n\to\infty} \|f-\varphi_n\|_\Spc X=0$.
\end{definition}



Also important for our purpose is the adjoint operator $\Op T^\ast: \Spc Y'\to \Spc X'$, which is the unique linear operator
such that 
$$\langle g,\Op T\{ f\}\rangle_{\Spc Y' \times \Spc Y}=\langle \Op T^\ast\{ g\},f\rangle_{\Spc X' \times \Spc X}$$
for any $g \in \Spc Y'$ and $f\in \Spc X'$, where the spaces $\Spc X'$ and $\Spc Y'$ are the duals of the topological vector spaces $\Spc X$ and $\Spc Y$, respectively.
If $\Op T: \Spc X \toC \Spc Y$ is bounded with operator norm $\|\Op T\|$, then the adjoint
$\Op T^\ast: \Spc Y' \toC \Spc X'$ is bounded with $\|\Op T^\ast\|=\|\Op T\|$. In particular,
the adjoint of the convolution operator $\Op T_h: \Spc D(\R) \toC \Spc D'(\R)$ is $\Op T_{h^\vee}: 
\Spc D(\R) \toC \Spc D'(\R)$, where $h^\vee$ is the time-reversed impulse response such that  $\langle h^\vee,\varphi \rangle=\langle h,\varphi^\vee \rangle$, where 
$\varphi^\vee(t)\eqdef\varphi(-t)$.

We now briefly show how we make use of these two mechanisms to specify the continuous extension $\Op T_h: L_\infty(\R) \to L_\infty(\R)$ with $h \in \Spc M(\R)$ (or, $h \in L_1(\R)$) that is implicitly referred to in Theorems \ref{Theo:BIBO2} and \ref{Theo:BIBO3}.
The enabling ingredient there is the continuity bound $\|h^\vee \ast \varphi\|_{L_1} \le \|h^\vee\|_{\Spc M} \|\varphi\|_{L_1}$ (see proof of Theorem \ref{Theo:Convolution}, Item 4), which also yields $h^\vee \ast \varphi \in L_1(\R)$ for all $\varphi \in \Spc D(\R)$. We then apply Definition \ref{Def:Convolution}
to specify the unique extension $\Op T_{h^\vee}: L_1(\R) \toC L_1(\R)$. An important point for our argumentation is that this (pre-adjoint) convolution operator also has a concrete implementation as
\begin{align}
\label{Eq:ConvMeasureaAdjoint}
t \mapsto (h^\vee \ast \varphi) (t)= \langle h^\vee, \varphi(t-\cdot)\rangle=\int_\R \varphi(\tau+t) \dint \mu_h(\tau),
\end{align}
which is supported by the same continuity bound with $\varphi$ now ranging over $L_1(\R)$
instead of the smaller  space $\Spc D(\R)$.
The existence and uniqueness of $\Op T_{h^\vee}: L_1(\R) \toC L_1(\R)$ then guarantees the existence and unicity of the adjoint $\Op T^\ast_{h^\vee}: L_\infty(\R) \toC L_\infty(\R)$.
To show that $\Op T^\ast_{h^\vee}=\Op T_{h}$, we use the explicit representation of $\Op T_{h^\vee}$ given by \eqref{Eq:ConvMeasureaAdjoint} with $h^\vee \in \Spc M(\R)$ and invoke Fubini's theorem to justify the interchange of integrals in
\begin{align*}
\langle\Op T_{h^\vee}\{f\}, g \rangle&=\int_\R \left(\int_\R  f(\tau+t) \dint \mu_h(\tau)\right) g(t) \dint t\\
& 
=\int_\R \int_\R  f(x)  g(x-\tau) \dint \mu_h(\tau) \dint x\\
& =\int_\R f(x) \left( \int_\R  g(x-\tau) \dint \mu_h(\tau)\right) \dint x\\
&=\langle f, \Op T_h\{g\} \rangle
\end{align*}
for any $f \in L_1(\R)$ and $g \in L_\infty(\R)$.
This proves that the original convolution operator defined by
\eqref{Eq:ConvMeasure} coincides with the adjoint of
$\Op T_{h^\vee}: L_1(\R) \toC L_1(\R)$,
which is also consistent with the property $h=(h^\vee)^\vee $.
Since $\Spc D(\R) \subset L_\infty(\R)$, we can therefore uniquely identify $\Op T_h: L_\infty(\R) \toC L_\infty(\R)$ as the extension
of $\Op T_h: \Spc D(\R) \to \Spc D'(\R)$ that preserves the adjoint relation $\Op T^\ast_{h^\vee}=\Op T_{h}$. 

\subsection{Proof of Banach Variants of BIBO Stability}
The Banach spaces of interest for us are
$\Spc X=C_0(\R),L_p(\R)$ and $\Spc Y=C_{0}(\R), C_{\rm b}(\R), L_p(\R)$ with $p\ge1$.

\begin{theorem} 
\label{Theo:Convolution}
Depending on the functional properties of its impulse response $h\in \Spc D'(\R)$, the convolution operator $\Op T_h: \Spc D(\R) \toC \Spc D'(\R)$ defined by  \eqref{Eq:ConvD} admits the following (unique) 
continuous extensions\footnote{See Definition \ref{Def:Convolution} and accompanying explanations. The bottom line is that the definition of these operators is compatible with the convolution integral \eqref{Eq:Conv} or \eqref{Eq:ConvMeasure} depending on whether $h$ is a function or a Radon measure. }
\begin{enumerate}  \setlength{\itemsep}{4pt}%
\item Let $p,q\in(1,\infty)$ be conjugate exponents with $\frac{1}{p}+\frac{1}{q}=1$. Then,
$h \in L_q(\R) \  \Rightarrow \ \Op T_h: L_p(\R) \toC C_{0}(\R)$ with
$\|\Op T_h\|_{L_p \to C_{0}}\le \|h\|_{L_q}$. 

\item $h \in L_1(\R) \quad \Rightarrow  \quad \Op T_h: L_\infty(\R) \toC C_{\rm b}(\R)$
with $\|\Op T_h\|_{L_\infty \to C_{\rm b}}=\|h\|_{L_1}$.

\item $h \in \Spc M(\R) \quad \Leftrightarrow\quad \Op T_h: C_0(\R) \toC C_{\rm b}(\R)$.
\item $h \in \Spc M(\R) \quad \Leftrightarrow \quad \Op T_h: L_{1}(\R) \toC L_{1}(\R)$. 
\item $h \in \Spc M(\R) \quad \Leftrightarrow \quad \Op T_h: L_\infty(\R) \toC L_\infty(\R)$. 
\end{enumerate}
Moreover, the operator norms for Items 3-5, characterized by an equivalence relation, are 
$\|\Op T_h\|_{C_0\to C_{\rm b}}=\|\Op T_h\|_{L_1\to L_1}=\|\Op T_h\|_{L_\infty\to L_\infty}=\|h\|_{\Spc M}$.
Finally, under the hypothesis of local integrability
$h\in L_{1,\rm loc}(\R)$, the continuity of $\Op T_h: L_\infty(\R) \toC C_{\rm b}(\R)$ implies that $h \in L_1(\R)$, which is the converse part of Item 2.

\end{theorem}
\begin{IEEEproof}

{\em Item 1}.
Under the assumption that $h \in L_q(\R)$ with $q\ge 1$,  we invoke H\"older's inequality
\begin{align*}
\left|(h \ast \varphi)(t)\right|\le\int_\R |h(\tau)|\cdot |\varphi(t-\tau)| \dint \tau\le \|h\|_{L_q}  \|\varphi(\cdot-t)\|_{L_p} 
\end{align*}
for any $\varphi \in \Spc D(\R)$, which yields the required upper bound
 ${\|h \ast \varphi\|_{L_\infty}} \le \|h\|_{L_q}  \|\varphi\|_{L_p}$. Likewise, by linearity, we get that
\begin{align*}
\left|(h \ast \varphi)(t)-(h \ast \varphi)(t-\Delta t)\right|&=
\left|h \ast \big(\varphi(t)-\varphi(t-\Delta t)\big)\right| \\
& \le \|h\|_{L_q} \cdot \|\varphi-\varphi(\cdot-\Delta t)\|_{L_p}.
\end{align*}
Due to the constraining topology of $\Spc D(\R^d)$, 
$\lim_{\Delta t\to 0}\|\varphi-\varphi(\cdot-\Delta t)\|_{L_p}=0$ for any $p\ge1$,
which 
proves the continuity of the 
function $t \mapsto {(h \ast \varphi)(t)}$. This leads to the intermediate outcome $h \ast \varphi \in C_{\rm b}(\R)$ for all $\varphi \in \Spc D(\R)$. 

If we now replace $h$ by $\phi \in \Spc D(\R)$, we readily deduce that
$\Op T_\phi\{\varphi\}=\phi \ast \varphi$ is compactly supported; hence, $\Op T_\phi\{ \varphi \}\in C_0(\R)$ for all $\varphi \in \Spc D(\R)$ with ${\|\phi \ast \varphi\|_{L_\infty}}\le \|\phi\|_{L_q}\|\varphi\|_{L_p}$. We then invoke
Theorem \ref{Theo:BLTExtension} with $\Spc X=\overline{(\Spc D(\R),\|\cdot\|_{p})}$
to deduce that
$\Op T_\phi: L_p(\R) \toC C_0(\R)$ for $p\in [1,\infty)$ and $\Op T_\phi: C_0(\R) \toC C_0(\R)$ for any $\phi \in \Spc D(\R)$.
Since the convolution is commutative, this implies that $\phi \ast h=h \ast \phi \in C_0(\R)$ for any $h\in L_q(\R)$ with $q\in(1,\infty)$
\big(resp., $h \in C_0(\R)$\big)
and $\phi \in \Spc D(\R)\subset L_p(\R)$
which, by completion with respect to the $\|\cdot\|_{L_p}$ norm with $p\in(1,\infty)$, gives $\Op T_h: L_p(\R) \toC C_0(\R)$
with $\|\Op T_h\|_{L_p \to C_0}\le \|h\|_{L_q}$
(resp., $\Op T_h: C_0(\R) \toC C_0(\R)$ with $\|\Op T_h\|_{C_0 \to C_0}= \|h\|_{L_1}$).


{\em Item 3}. Since $C_{\rm b}(\R) \embedIso L_\infty(\R)$, the relevant duality bound there is \eqref{Eq:Minequal}, which yields $\|h \ast \varphi\|_{L_\infty}\le \|h\|_{\Spc M}\; \|\varphi\|_{L_\infty}$. This allows us to use the same argument as in Item 1 to show that
$\Op T_h\{\varphi\} \in C_{\rm b}(\R)$ for all $\varphi \in \Spc D(\R)$.
 Since $C_0(\R)=\overline{(\Spc D(\R^d),\|\cdot\|_{L_\infty})}$, we then apply the proven completion technique to specify the unique operator
$\Op T_{h}: C_0(\R) \to C_{\rm b}(\R)$ with
  $\|\Op T_h\|_{C_0 \to C_{\rm b}}\le \|h\|_{\Spc M}$. 
Conversely, let $\Op T_h: C_0(\R) \toC C_{\rm b}(\R)$ with operator norm $\|T_h\|_{C_0 \to C_{\rm b}}<\revise{\infty}$.
Then, for any $\varphi \in C_0(\R)$, $$(h \ast \varphi)(0)=\langle h, \varphi^\vee\rangle \le \|T_h\|_{C_0 \to C_{\rm b}}\; \|\varphi\|_{L_\infty}$$ with $\varphi^\vee \in C_0(\R)$ 
and $\|\varphi^\vee\|_{L_\infty}=\|\varphi\|_{L_\infty}$. By substituting $\varphi$ for $\varphi^\vee$ and by recalling that $\Spc D(\R)$ is dense in $C_0(\R)$, we get that
\begin{align}
\sup_{\varphi \in C_0(\R)\backslash\{0\}} 
\frac{\langle h, \varphi\rangle}{\|\varphi\|_{L_\infty}}&=\sup_{\varphi \in \Spc D(\R)\backslash\{0\}} 
\frac{\langle h, \varphi\rangle}{\|\varphi\|_{L_\infty}} \nonumber \\
&=\|h\|_{\Spc M} \le \|\Op T_h\|_{C_0 \to C_{\rm b}},
\end{align}
which then also proves that the bound is sharp.

{\em Item 4}. The key here is the estimate 
\begin{align*}
\int_\R \big|(h \ast f) (t)\big|\, \dint t &\le \int_\R \int_\R |f(t-\tau)| \, \dint|\mu_h|( \tau)\,\dint t  \\
&=\int_\R \left(\int_\R |f(x)|\dint x \right)  \dint|\mu_h|( \tau) \tag{by Fubini}
\\&= \left(\int_\R |f(x)| \dint x \right) \left(\int_\R \dint |\mu_h|\right)
\\&= \|f\|_{L_1}\, \|h\|_{\Spc M},
\end{align*}
from which we deduce
the boundedness of $\Op T_h: L_1(\R) \to L_1(\R)$ with
$\|\Op T_h\|_{L_1 \to L_1}\le \|h\|_{\Spc M}$. 
(The extension technique is essentially the same as in Item 1 with $p=1$ and $L_1(\R)=\overline{(\Spc D(\R),\|\cdot\|_{L_1})}$.) The converse
implication and the sharpness of the bound will be deduced from Item 5 by duality.

{\em Item 5}. Since $L_\infty(\R)=\big(L_1(\R)\big)'$, the adjoint of $\Op T_h: L_1(\R) \toC L_1(\R)$ is $\Op T_h^\ast=\Op T_{h^\vee}: L_\infty(\R) \toC L_\infty(\R)$.
The equivalence $h\in \Spc M(\R) \Leftrightarrow h^\vee \in \Spc M(\R)$ implies
the continuity of $\Op T_{h}: L_\infty(\R) \toC L_\infty(\R)$ with
$\|\Op T_h\|_{L_\infty \to L_\infty}\le \|h\|_\Spc M=\|h^\vee\|_{\Spc M}$. 
As for the converse implication, we take advantage of the embedding
$C_0(\R)  \embedIso L_\infty(\R)$, which allows us to reuse the argument
of Item 3.


{\em Item 2 and Its Converse}. The first part follows from the beginning of the proof of Item 1, the application of the extension principle for $p=1$, and the commutativity of the convolution integral, which yields $f \ast h= h \ast f \in C_{\rm b}(\R)$ with $\|h \ast f\|_{L_{\infty}}\le \|h\|_{L_1} \|f\|_{L_\infty}$ for any $f \in L_\infty(\R)$. We show that the bound is sharp by applying the convolution operator to the ``worst-case'' signal $f_0$ identified in \eqref{Eq:f0}.
Conversely, let $\Op T_h: L_\infty(\R) \toC C_{\rm b}(\R)$ with $\|T_h\|_{L_\infty \to C_{\rm b}}<\revise{\infty}$. Taking advantage of the isometric embedding $
C_{\rm b}(\R)  \embedIso L_\infty(\R)$, we then invoke the equivalence in Item 5 to deduce that
$\|h\|_{\Spc M}<\revise{\infty}$, which implies that $h\in \Spc M(\R)$. The announced equivalence then follows from Proposition \ref{Theo:BIBOcriterion} in Section \ref{Sec:Radon}.
%
\end{IEEEproof}

The result in Item 1 is discussed in most advanced treatises on the Fourier transform (e.g., \cite[Proposition 8.8, p 241]{Folland2013}).
We are including it here in a self-contained form---at the expense of a few more lines in the proof of Item 2---because it nicely characterizes the regularization effect of convolution.
The equivalences stated in Item 4 and Item 5 are known in the context of the theory of 
$L_p$ Fourier multipliers \cite[Section 2.5]{Grafakos2004}, even though the latter does not seem to have permeated to the engineering literature. 
\revise{The equivalence in Item 4 may also be identified as a special instance of Wendel's theorem in the abstract theory of multipliers on locally compact Abelian groups \cite[Theorem 0.1.1, p. 2]{Larsen1970}.}
Interestingly, the condition $h\in \Spc M(\R)$ is also
sufficient for the continuity of $\Op T_h: L_p(\R) \toC L_p(\R)$, a claim that is supported
by the Young-type norm inequality
\begin{align}
\label{Eq:Youngp}
\|h \ast f\|_{L_p} \le \|h\|_{\Spc M} \, \|f\|_{L_p},
\end{align}
which holds for any $f \in L_p(\R)$ with $p\ge1$. However, \eqref{Eq:Youngp} is only sharp at the two end points $p=1,+\infty$, in conformity with the statements in Items 4 and 5.
In fact, the only other case where the complete class of convolution operators $\Op T_h: L_p(\R) \toC L_p(\R)$ has been characterized is for $p=2$, with the necessary and sufficient condition being $\widehat h \in L_\infty(\R)$ (bounded frequency response) \cite[Theorem 3.18, p. 28]{Stein1971}, which is slightly more permissive than the BIBO requirement.
Indeed, $h \in \Spc M(\R)\Rightarrow \widehat h \in L_\infty(\R)$, whereas the reverse implication does not hold.

We like to single out Item 3 in Theorem \ref{Theo:Convolution}
as the pivot point that facilitates the derivation of the (nontrivial) reverse 
implications---namely, the necessity of the condition $h\in \Spc M(\R)$.
While the listed property is sufficient for our purpose, we can refer to a recent characterization by Feichtinger \cite[Theorem 2, p. 499]{Feichtinger2017b} which, in the present context, translates into the refined statement  ``$h \in \Spc M(\R) \Leftrightarrow \Op T_h: C_0(\R) \toC C_0(\R)$.'' 
The additional element there is the vanishing of
$(h \ast f)(t)$ at infinity, which calls for a more involved proof.

%
%
%

While the statement in Item 2 is a special case of Item 5, as made explicit in Section \ref{Sec:Radon}, the interesting part of the story is that
this restriction induces a smoothing effect on the output, ensuring that the function 
$t \mapsto (h \ast f)(t)$ is continuous. There is obviously no such effect for the
case $h=\delta \in \Spc M(\R)$ (identity) or, by extension, $h_{\rm d}=\sum_{n\inZ} a[n]\delta(\cdot-n)\in \Spc M(\R)$
with $\|h_d\|_{\Spc M}=\|a\|_{\ell_1}$, which corresponds to the continuous-time transposition of a digital filter.
\subsection{Explicit Criterion for BIBO Stability}
\label{Sec:Radon}
We now show how to determine $\|h\|_{\Spc M}$ (our extended criterion for BIBO stability) under the assumption that 
$h \in L_{1, {\rm loc}}(\R)$.
Any such impulse response  can be identified with a distribution by considering the linear form
\begin{align}
\label{Eq:DistributionD}
h: \varphi \mapsto \langle h, \varphi  \rangle = \int_\R h(t) \varphi(t)\dint t,
\end{align}
which continuously maps $\Spc D(\R)\to \R$. 
\Revise{It turns out that the latter is a special instance of a real-valued Radon measure, which is an extended type of measure whose $\|\cdot\|_{\Spc M}$-norm is not necessarily finite.}
\begin{definition}[{\cite{Schwartz:1966}}]
\label{Def:Radon}
\Revise{
A distribution
$f \in \Spc D'(\R)$ 
is called a real-valued Radon measure if, for any compact subset $\mathbb{K} \subset \R$, there exists a constant
$C_\mathbb{K}>0$ such that
\begin{align}
\label{Eq:RadonBound}
\langle f, \varphi \rangle\le  C_\mathbb{K} \sup_{t \in \mathbb{K}} | \varphi(t)|
\end{align}
for all $\varphi\in \Spc D(\mathbb{K})=\big\{\varphi \in \Spc D(\R): \varphi(t)=0, \forall t\notin \mathbb{K}\big\}$.}

\Revise{
A distribution 
$f^+ \in \Spc D'(\R)$ 
is said to be positive if
$\langle f^+,\varphi \rangle\ge 0$ for all 
$\varphi \in \Spc D^+(\R)=\big\{\varphi \in \Spc D(\R): \varphi(t)\ge 0, t \inR\big\}$. } 
\end{definition}
\Revise{The connection between the two kinds of distributions 
in Definition \ref{Def:Radon} is that a positive distribution is a special instance of a Radon measure,
while any real-valued Radon measure $f$ admits a unique decomposition as $f=(f^+-f^-)$, where both $f^+, f^-\ge 0$ are positive distributions  \cite[Theorem 21.2, p. 218]{Treves2006}.
One then also defines the corresponding ``total-variation measure''
$|f|=f^+ +f^-$, which is positive by construction.}

\Revise{It turns out that the Dirac impulse $\delta$ is a positive Radon measure with a universal bounding constant $C_\mathbb{K}=1$. Likewise, the minimal constant in \eqref{Eq:RadonBound} for $f\in L_{1,{\rm loc}}(\R)$ is
$\C_\mathbb{K}=\int_{\mathbb{K}}|f(t)|\dint t$, which is essentially what is expressed in Proposition \ref{Theo:BIBOcriterion}.}

\begin{proposition}[Total-variation norm for measurable functions]
\label{Theo:BIBOcriterion}
Let $h \in L_{1, {\rm loc}}(\R)$. Then, $\|h\|_{\Spc M}=\|h\|_{L_1}=\int_{-\infty}^{+\infty}|h(t)|\dint t$. Consequently,
$h \in \Spc M(\R)$ if and only if 
 $\int_{-\infty}^{+\infty}|h(t)|\dint t<\infty$.
\end{proposition}
\begin{IEEEproof}
\Revise{ 
In accordance with Definition \ref{Def:Radon}, we view 
$h \in L_{1, {\rm loc}}(\R)$ 
as a real-valued Radon measure
with $h^+(t)=\max\big(h(t),0\big)$ and $h^-(t)=\max\big(0,-h(t)\big)$,
while the corresponding total-variation measure is $|h|=h^++h^-\in L_{1, {\rm loc}}(\R)$ with $|h|:
t \mapsto |h(t)|$, which is consistent with the notation. We then distinguish between two cases.}

\Revise{ 
{\em (i) Bounded Scenario.} When
$\|h\|_{ \Spc M}<\infty$,  we can invoke the classical Jordan decomposition of a measure (see \cite{Folland2013}), 
\begin{align*}
\forall f\in \Spc M(\R):\ \|f\|_{\Spc M}=\|f^+\|_{\Spc M}+\|f^-\|_{\Spc M}=\||f|\|_{\Spc M}<\infty,
\end{align*}
which allows us to reduce the problem to the easier determination of $\||h|\|_\Spc M$.}
Accordingly, for any given $T>0$, we define  $h_T=|h| \cdot \Indic_{[-T,T]} \ge 0$ and observe that
\begin{align*}
\|h_T\|_{\Spc M}=\sup_{\varphi \in \Spc D(\R):\, \|\varphi\|_{L_\infty}\le 1} \langle h_T, \varphi \rangle&=\langle h_T, 1 \rangle\\
&=\|h_T\|_{L_1}< \infty,
\end{align*}
where the supremum is achieved by considering any test function $\varphi_T \in \Spc D(\R)$ such that $\varphi_T(t)=1$ for all $t\in [-T,T]$. \Revise{In the limit, we get that
$\lim_{T \to \infty} \|h_T\|_{L_1}=\lim_{T \to \infty} \|h_T\|_{\Spc M}=\| |h|\|_{\Spc M}<\infty$, from which we conclude that $\||h|\|_{\Spc M}=\|h\|_{\Spc M}=\|h\|_{L_1}$.}

\Revise{{\em (ii) Unbounded Scenario.}
The  condition $\|h\|_{\Spc M}=\infty$ can be formalized as: for any $n\in \N$, there exists  $\varphi_n\in \Spc D(\R)$
with $\|\varphi_n\|_{L_\infty}=1$ such that $\langle  h,\varphi_n\rangle> n$. However,
\begin{align*}
\langle  h, \varphi_n\rangle\le \big|\int_\R h(t) \varphi_n(t) \dint t\big|\le \int_\R |h(t)|\dint t \; \|\varphi_n\|_{L_\infty}=\|h\|_{L_1}.
\end{align*}
Therefore, $\|h\|_{L_1}>n$ for all $n \in \N$, leading to $\|h\|_{L_1}=\infty$.}


\end{IEEEproof}

Let us now conclude with a few more observations.

\revise{Since $L_{1,{\rm loc}}(\R)$ can be identified as the subspace of measures that are absolutely continuous (see \cite[p. 18]{Schwartz:1966}),
the result in Proposition \ref{Theo:BIBOcriterion} is consistent with the well-known property in probability theory that $L_1(\R)$ coincides with the
subset of bounded measures that are absolutely continuous.}


Under the minimalistic assumption that $h \in  L_{1, {\rm loc}}(\R)$, the convolution integral \eqref{Eq:Conv} is
well defined for any $t \inR$ provided that the input function $f: \R \to \R$ is {\em bounded} and {\em compactly supported}. 
Equation \eqref{Eq:Conv} then even yields an output function $t \mapsto (h \ast f)(t)$ that is continuous, as shown in Appendix D. However, the trouble comes from the fact that the output then inherits the potential lack of decay of $h$ when $h\notin L_1(\R)$.

One can also make a connection between the result in Proposition \ref{Theo:BIBOcriterion} and the standard argument that is presented to justify the necessity of $h \in L_1(\R)$.
When the latter condition is fulfilled, we have that
\begin{align}
\| h\|_{\Spc M}&=\sup_{\varphi \in \Spc D(\R):\, \|\varphi\|_{L_\infty}\le 1} \langle h, \varphi \rangle
\nonumber\\
&=\|h\|_{L_1}=\sup_{\phi \in L_\infty(\R):\, \|\phi\|_{L_\infty}\le 1} \langle h, \phi \rangle=\int_\R h(t)\phi_0(t)\dint t,
\end{align}
where $\phi_0(t)={\rm sign}\big(h(t)\big)$.
While the supremum is achieved exactly over $L_\infty(\R)$ by taking
$\phi=\phi_0$, it is a bit trickier over $\Spc D(\R)$ because of the additional smoothness requirement.
Yet, due to the definition of the supremum, for any $\epsilon>0$ there exists a function 
$\varphi_\epsilon \in \Spc D(\R)$ with $\|\varphi_\epsilon\|_{\infty}=1$
such that $\int h(t)\varphi_\epsilon (t)\dint t = (1-\epsilon) \|h\|_{\Spc M}\le \|h\|_{\Spc M}=\|h\|_{L_1}$. By taking $\epsilon$ arbitrarily small, we end up with 
$\varphi_\epsilon$ being a ``smoothed'' rendition of $\phi_0$, so that the spirit of the initial argument is retained.

\section*{Appendix}
\subsection*{A. Is the Dirac Distribution a Member of $L_1(\R)$?}
Let us start with the historical observation that the eponymous impulse $\delta$ is already present in the (early) works of both Fourier and Heaviside \cite{Komatsu2002}. The former, as one would expect, defined it via an ``improper'' integral (the inverse Fourier transform of ``$1$''), while the latter identified $\delta$ as the ``formal'' derivative of the unit step (a.k.a.\ the Heaviside function). However, the mathematics for giving a rigorous sense to these identifications were missing at the time; one had to wait for the development Schwartz' distribution theory in the 1950s \cite{Schwartz:1966}, which already shows that the mere process of obtaining a rigorous definition of $\delta$ was far from trivial.

From the pragmatic point of view of an engineer, the title question is at the heart of the matter to understand the scope of Proposition \ref{Theo:BIBO}, and the source of some confusion, too. 
Let us start by listing the elements that could suggest that the answer to the question is positive.
\begin{enumerate}
\item It is common practice to make liberal use of what mathematicians consider abusive notations; in particular, equations such as
$f(t)=\int_\R \delta(\tau)f(t-\tau)\dint \tau$, which could suggest that
$\delta(\tau)$ can be manipulated as if it were a classical function of $\tau$.
\item Dirac's $\delta$ has the unit ``integral'' $\langle \delta,1\rangle=1$,
which is indicated formally as $\int_\R \delta(\tau)\dint \tau=1$. Moreover, $\delta\ge 0$ in the sense that it is a positive distribution (see Definition \ref{Def:Radon}).
\item The Dirac impulse is often described as the limit of $\varphi_n(t)=\frac{n}{\sqrt{2 \pi}} \ee^{-(t n)^2/2}$ as $n\to \infty$, with $\varphi_n\in \Spc S(\R)$. Since $\|\varphi_n\|_{L_1}=1$ for any $n>0$, this could suggest that $\lim_{n\to \infty}\|\varphi_n\|_{L_1}=1$ as well.
\end{enumerate}
In order to convince the reader that the answer to the title question is actually negative, we now refute these intuitive arguments one by one.
\begin{enumerate}
\item The explicit description of the Dirac impulse as a centered Gaussian distribution whose standard deviation $\sigma_n=1/n$ tends to zero suggests that
$\delta=\lim_{n\to \infty} \varphi_n$ must be entirely localized at $t=0$.
%
%
%
The best attempt at describing this limit in Lebesgue's world of measurable functions would be
$$
p_0(t)=\left\{\begin{array} {ll}
+\infty,&  t=0\\ 0,& \mbox{otherwise},
\end{array}\right.
$$
which is equal to zero almost everywhere.
However, since the width of the impulse is zero, we get that $\int_\R p_0(t)\dint t=0$, which is incompatible with the property that $\int_\R \delta(t)\dint t=1$. This points to the impossibility of representing $\delta$ by a function in $L_1(\R)$ or even in $L_{1,{\rm loc}}(\R)$.
Strictly speaking, $\delta$ is defined as a continuous linear functional on $\Spc D(\R)$---or, by extension, $C_0(\R)$---which precludes the application of any nonlinear operation (such as  $|\cdot|^p$) to it.
\item The generalized Fourier transform of $\delta$ is $\Fourier\{\delta\}=1$, which is  bounded, but not decreasing at infinity.
If $\delta$ was included in $L_1(\R)$, this would contradict the Riemann-Lebesgue Lemma, which is equivalent to $\Fourier: L_1(\R) \toC C_0(\R)$ with $\|\Fourier\|_{L_1 \to C_0}=1$.
By contrast, the inclusion $\delta\in \Spc M(\R)$ is compatible with the (dual) continuity property
of the Fourier transform $
\Fourier^\ast, \Fourier: \Spc M (\R) \toC L_\infty(\R)
$ with $\|\Fourier^\ast\|_{\Spc M \to L_0}=1$.
\item While the sequence of rescaled Gaussians $(\varphi_n)$ converges to $\delta\in \Spc S'(\R) \embedC \Spc D'(\R)$ in the (weak) topology of $\Spc S'(\R)$ (Schwartz' space of tempered distributions), the problem is that it fails to be a Cauchy sequence in the (strong) norm topology of $L_1(\R)$. Hence, there is no guarantee that $\delta=\lim_{n\to \infty}\varphi_n$ stays in $L_1(\R)$.
\end{enumerate}

\subsection*{B. Examples of Inaccurate Statements on BIBO Stability}
This list is far from exhaustive and not intended to downplay the important contributions of the listed people who are internationally recognized leaders in the field. Its sole purpose is to illustrate the omnipresence of the misconception in the engineering literature, including in some of the most popular and authoritative textbooks in the theory of linear systems and signal processing.

As a start, one can read in the English version of Wikipedia that a necessary and sufficient condition for the BIBO stability of a convolution operator is that its impulse response be absolutely integrable, formulated as
$\int_\R |h(\tau)| \dint \tau =\|h\|_{L_1}< \infty$. In view of the discussion around Proposition \ref{Theo:BIBO}, this is only correct if one restricts the scope of the statement to those impulse responses that are Lebesgue-measurable and locally integrable. 

Kailath mentions in \cite[p. 175]{Kailath1980} that the equivalence between BIBO stability and $h\in L_1(\R)$ is well known, and attributes the result to James, Nichols, and Phillips \cite{James1947}.  It turns out that the pioneers of the theory on control and linear systems were focusing their attention on analog systems ruled by ordinary differentiable equations whose impulse responses are sums of causal exponentials and, therefore, Lebesgue-measurable. 
Kailath then presents a proof on p.\ 176 that is essentially the one we used for
Proposition \ref{Theo:BIBO}, except that he neither considers a limit process nor explicitly says that $h$ must be (locally) integrable.

Oppenheim and Willsky discuss the property in \cite[p. 113-114]{Oppenheim:1996}. To justify the BIBO stability of the pure time-shift operator (including the identity), they then present an argument in support of the inclusion of $\delta(\cdot-t_0)$ in $L_1(\R)$ (Example 2.13) which, in view of the discussion in Appendix A, is flawed.

Vetterli {\em et al.} claim in \cite[Theorem 4.8, p. 357]{Vetterli2014} that the operator $\Op T_h$ is BIBO-stable from $L_\infty(\R) \to L_\infty(\R)$ if and only if $h\in L_1(\R)$, a statement that is incompatible with Theorem \ref{Theo:BIBO3}. This can be corrected by limiting the scope of the equivalence as in the statement of Theorem \ref{Theo:BIBO2}.

\Revise{
\subsection{Convolution in the ``Unstable'' Scenario} 
\label{App:Unstable}
Here, we characterize the output of a potentially ``unstable'' filter when the input signal is compactly supported. The enabling hypothesis is the local integrability of the impulse response.
\begin{proposition}
\label{Prop:Loc}
Let $f\in L_\infty(\R)$ be compactly supported and $h \in L_{1,{\rm loc}}(\R)$. Then, the function $t \mapsto (h \ast f)(t)$ defined by 
\eqref{Eq:Conv} is bounded on any compact set $\mathbb{K} \subset \R$ and continuous; that is,  $h \ast f \in C(\R)$. 
\end{proposition}
\begin{proof}
Because of the local integrability of $h$, the convolution integral \eqref{Eq:Conv} is well-defined for any $t \inR$ with
\begin{align*}
(h \ast f)(t)=\int_{\R} h(\tau)  f(t-\tau) \dint \tau=\int_{\mathbb{M}} f(x) h(t-x)  \dint x \tag{by change of variable}
\end{align*}
and
\begin{align*}
|(h \ast f)(t)|\le 
\|f\|_{L_\infty} \int_{\mathbb{M}} |h(t\pm\tau)|  \dint \tau<\infty,
\end{align*}
where $\mathbb{M}$ is the smallest symmetric interval such that $f(t)=f(-t)=0$ for all $t\notin \mathbb{M}$.
For any given open bounded set $\mathbb{K}\subset \R$, we then observe that
\begin{align*}
\forall t \in \mathbb{K}: \quad (h \ast f)(t)=(h_{\mathbb{K}+\mathbb{M}} \ast f)(t)
\end{align*}
where $h_{\mathbb{K}+\mathbb{M}}=h\cdot \One_{\mathbb{K}+\mathbb{M}}$ is the restriction of
the original impulse response to the set $\mathbb{K}+\mathbb{M}=\{t +\tau: t \in \mathbb{K}, \tau \in \mathbb{M}\}$.
Since $h_{\mathbb{K}+\mathbb{M}}\in L_1(\R)$, one has that
\begin{align}
\sup_{t \in \mathbb{K}} |h \ast f(t)|\le 
\|f\|_{L_\infty} \|h_{\mathbb{K}+\mathbb{M}}\|_{L_1}<\infty.
\end{align}
Likewise, for any $t, t_0 \in \mathbb{K}$, we have that 
\begin{align}
|h \ast f(t)&-h \ast f(t_0)| \nonumber\\
& \le \|f\|_{L_\infty}    \|h_{\mathbb{K}+\mathbb{M}}(t-\cdot)-h_{\mathbb{K}+\mathbb{M}}(t_0-\cdot)\|_{L_1}.
\label{Eq:contbound}
\end{align}
Next, we invoke Lebesgue's dominated-convergence theorem and the property that $C_0(\R)$ is dense in $L_1(\R)$ to show that 
$\|h_{\mathbb{K}+\mathbb{M}}(t-\cdot)-h_{\mathbb{K}+\mathbb{M}}(t_0-\cdot)\|_{L_1}\to 0$ as $t \to t_0$. 
This, together with \eqref{Eq:contbound},
implies that $\lim_{t\to t_0}|(h \ast f)(t)-(h \ast f)(t_0)|=0$, which expresses the continuity of $t \mapsto h \ast f(t)$ at $t=t_0$ for any $t_0\in \mathbb{K}$.  
\end{proof}
}
\section*{Acknowledgments}
The author is extremely thankful to Julien Fageot, Shayan Aziznejad and Hans Feichtinger
for having spotted inconsistencies in earlier versions of the manuscript and 
for their  helpful advice. \revise{He is also appreciative of 
Thomas Kailath and Vivek Goyal's feedback, as well as that of the three anonymous reviewers}.

\bibliographystyle{plain}
%
%
%
%
%
\bibliography{BIBO}

\end{document}